\documentclass{article}[12pt]

\usepackage[usenames]{color}

\usepackage[dvipsnames]{xcolor}
\usepackage{amsmath,amsthm,amsfonts}
\usepackage{amsfonts,amsmath,latexsym}
\usepackage{graphicx}
\usepackage{epsfig}
\usepackage{color}
\usepackage{comment}
\newtheorem{theorem}{Theorem}

\newtheorem{lemma}[theorem]{Lemma}
\newtheorem{remark}[theorem]{Remark}

\newcommand{\bv}{{\bf w}}

\newcommand{\bF}{{\bf F}}
\newcommand{\bbs}{\bar{\boldsymbol\sigma}}
\newcommand{\beu}{\bar e_u}
\newcommand{\teu}{\tilde \phi_u}

\newcommand{\bes}{\bar e_{\boldsymbol\sigma}}
\newcommand{\tes}{\bar \phi_{\boldsymbol\sigma}}

\newcommand{\bs}{{\boldsymbol\sigma}}

\newcommand{\bV}{{\bf W}}

\newcounter{bean}

\newcommand{\Ba}{\partial_t^{\alpha}}
\newcommand{\I}{\mathcal I}

\newcommand{\dop}{\partial_\tau^{\alpha}}
\newcommand{\dopm}{\partial_\tau^{-\alpha}}
\newcommand{\mdop}{\partial_\tau^{\alpha-1}}

\newcommand{\dbeu}{\bar {\cal E}_u }
\newcommand{\dteu}{\bar \Phi_u}

\newcommand{\dbes}{\bar {\cal E}_\bs}
\newcommand{\dtes}{\bar \Phi_\bs}

\title{A mixed FEM for a time-fractional Fokker-Planck model\thanks{The work of Naveed Ahmed has been funded through the GUST Internal Seed Grant, case number 8. The work of Samir Karaa was supported by Sultan Qaboos University under Grant IG/SCI/MATH/21/01.}}
\author{Samir Karaa\thanks{Department of Mathematics, Sultan Qaboos University, Al-Khod 123, Muscat, Oman, Email: skaraa@squ.edu.om}, Kassem Mustapha\thanks{School of Mathematics and Statistics, University of New South Wales, Sydney, Australia, Email: kassem.ahmad.mustapha@gmail.com}, Naveed Ahmed\thanks{Center for Applied Mathematics and Bioinformatics (CAMB), Department of Mathematics and Natural Sciences, Gulf University for Science and Technology, Hawally,  Kuwait, Email:ahmed.n@gust.edu.kw}}
\date{\today}

\begin{document}

\maketitle

\begin{abstract}
We propose and analyze a mixed finite element method for the spatial approximation of a time-fractional Fokker--Planck equation in a convex polyhedral domain, where the given driving force is a function of space. Taking into account the limited smoothing properties of the model, and considering an appropriate splitting of the errors, we employed a sequence of clever energy arguments to show optimal convergence rates with respect to both approximation properties and regularity results. In particular, error bounds for both primary and secondary variables are derived in $L^2$-norm for cases with smooth and nonsmooth initial data. We further investigate a fully implicit time-stepping scheme based on a convolution quadrature in time generated by the backward Euler method. Our main result provides pointwise-in-time optimal $L^2$-error estimates for the primary variable. Numerical examples are then presented to illustrate the theoretical contributions.
\end{abstract}


\section{Introduction}
Consider a bounded convex polygonal domain $\Omega$ in $\mathbb{R}^2$ with a boundary $\partial \Omega$, and let $T>0$ be a fixed value. The primary objective of this paper is to develop and analyze a mixed finite element method (FEM) for solving the  time-fractional Fokker-Planck model problem \cite{Metzler1999,HLS-2010}: with $0<\alpha<1,$ 
\begin{equation}\label{main}
\partial_t u-\nabla\cdot\bigl(
	^R \partial_t^{1-\alpha} \kappa \nabla u
	-\bF\, ^R\partial_t^{1-\alpha}u\bigr) =0,\quad\mbox{ in }\Omega\times (0,T],
\end{equation}
subject to homogeneous Dirichlet boundary conditions, with $u(x,0)=u_0(x)$ at the initial time $t=0.$ Here $\partial_t = \partial/\partial t$, the time-independent diffusivity $\kappa \in L^\infty(\Omega)$ and bounded below by $c_0>0$. In \eqref{main},  the Riemann-Liouville fractional derivative $^R \partial_t^{1-\alpha}=\frac{d}{dt} \I^{\alpha}$, where the time-fractional integral 
\begin{equation*} 
\I^{\alpha}{\varphi}(t):=\int_0^t\omega_\alpha(t-s)\varphi(s)\,ds\quad\text{with} \quad
\omega_{\alpha}(t):=\frac{t^{\alpha-1}}{\Gamma(\alpha)}.
\end{equation*}
We assume that the time-independent driving force~$\bF$ and its divergence  $\nabla \cdot \bF$ are bounded in $\Omega$. Since $\bF$ is independent of~$t$, applying $\I^{1-\alpha}$ to ~\eqref{main} and using $\I^{1-\alpha}(^R\partial_t^{1-\alpha} u) =u$ 
yield the alternative formulation
\begin{equation}\label{a1}
\Ba u + {\cal L} u =0 \quad\mbox{ on }\Omega\times (0,T],
\end{equation}
where the second-order non-selfadjoint operator ${\cal L} u=-\nabla\cdot (\kappa \nabla u)+ \nabla\cdot (\bF u)$. In \eqref{a1}, $\Ba u(t):=\I^{1-\alpha}{\partial_t u}(t)$ is the
time-fractional Caputo derivative of order $\alpha$ ($0<\alpha<1$). As $\alpha$ approaches $1^-$, $\Ba u$ converges to $\partial_tu$, and consequently, problem \eqref{main} (equivalently \eqref{a1}) reduces to the standard Fokker-Planck equation. For a comprehensive analysis of the well-posedness of the model problem \eqref{main} with a general driving force $\bF$, we refer the reader to \cite{McLeanetal2019}.

Various authors have studied the numerical solution of \eqref{a1}, mostly for a
1D spatial domain and a smooth continuous solution (which is not practically the case). For instance, Deng \cite{Deng-2007} transferred \eqref{a1} to a system of fractional ordinary differential equations (ODEs) by discretizing the spatial derivatives, and then applied a predictor-corrector approach combined with the method of lines. Cao et al. \cite{CFH-2012} adopted a similar approach and solved the resulting system using a second-order backward Euler scheme. 
Chen et al. \cite{SLZA-2009} investigated three different implicit finite difference techniques, in each of which the diffusion term was approximated by the standard 
second-order difference approximation at the advanced time level. In contrast, Jiang \cite{Jiang-2015} established monotonicity properties of the numerical solutions obtained by using these schemes, and so showed that the time-stepping preserves the non-negativity of the solution. Jiang and Xu \cite{JiangXu} proposed a finite volume method, and Yang et al. \cite{Yangetal} a spectral collocation method. Fairweather et al.~\cite{FZYX-2015} developed an orthogonal spline collocation method in space combined with the backward Euler method in time. Saadmandi \cite{SDA-2012} studied a collocation method based on shifted Legendre polynomials in time and Sinc functions in space. Vong and Wang \cite{VW-2015} have analyzed a high-order, compact difference scheme. Cui \cite{Cui-2015} applied a high-order approximation for the time-fractional derivative combined with a compact exponential finite difference scheme. 

For problem \eqref{main} with a time-space dependent $\bF$, first-order backward Euler and second-order corrected Crank-Nicolson time-stepping schemes were studied in \cite{LMM-2016,HLS-2021} and \cite{Mustaphaatel2023}, respectively. The stability and error analyses were carried out under certain regularity assumptions on the given data excluding the case of nonsmooth initial data, this is in addition to some restrictions on the range of the fractional exponent $\alpha.$ For the spatial discretization, a continuous, piecewise-linear Galerkin FEM was considered in a series of papers, see the recent papers \cite{LMM-2018,McLeanMustapha2022} and related references therein. In \cite{McLeanMustapha2022}, based on novel energy arguments, $\alpha$-robust stability, and optimal convergence results were investigated for both cases of smooth and nonsmooth initial data.

In the special case $\bF \equiv  0$ in \eqref{main}, mixed FEM were studied only in \cite{ZCBLT2017,K-2018,K-2020,KP-2020}. In \cite{ZCBLT2017}, the authors established mixed finite element error analysis under higher (unpractical) regularity assumption on the continuous solution $u$. Subsequently, Karaa \cite{K-2018,K-2020} has exploited the inverse of two mixed operators and derived optimal error estimates allowing nonsmooth initial data. The studies included both semidiscrete and fully discrete schemes. Later on, an alternative approach was adopted in \cite{KP-2020} to handle the case of the time-dependent diffusivity coefficient. The error analyses in these papers are not applicable for the case of non-zero $\bF$, and extending it to cover this case is not feasible.

The main aim of this work is: (1) develop a spatial mixed FEM for problem \eqref{a1} and derive $\alpha$-robust (that is, remains valid as $\alpha$ approaches $1$) error estimates which are not only optimal in approximation property but also in regularity of the solution; (2) combine our mixed FEM with a convolution quadrature generated by the backward Euler method in time and show optimal error estimates for the primary variable. Our convergence analysis approach involves splitting the errors into two distinct components. Then, a suitable exploitation of the results~\cite{K-2018} and~\cite{K-2020}, and the use of a sequence of challenging energy arguments in combination with general weakly singular Gronwall inequalities  (see \cite[Theorem 3.1]{DixonMcKee} and  \cite[Theorem 6.1]{DixonMcKee} for the continuous version and discrete version, respectively) allow us to prove the required optimal error bounds for both smooth and nonsmooth initial data. 

The paper is organized as follows. In Section \ref{sec:WRT}, we introduce the necessary notations and gather some technical estimates involving fractional integrals. In Section \ref{sec:Mixed}, we define the semidiscrete mixed FEM for solving \eqref{a1}. Error estimates follow in Theorem~\ref{thm:mixed-sm} showing optimal convergence rates for both primary and flux variables. We combine our mixed FEM with a time-stepping scheme that is based on a convolution quadrature in time generated by the backward Euler method in Section \ref{MFEMC}. A pointwise-in-time optimal $L^2$-error estimate is proved in Theorem~\ref{thm:mixed-smd}. Finally, in Section \ref{Numerical}, numerical experiments are presented to confirm our theoretical findings.

\section {Preliminaries}\label{sec:WRT} 
We shall first introduce notation and recall some preliminary results. Let $(\cdot,\cdot)$ be the inner product in $L^2(\Omega)$ and $\|\cdot\|$ the induced norm. For a nonnegative integer $m$, the space $H^m(\Omega)$ denotes the standard Sobolev space with the usual norm $\|\cdot\|_{H^m(\Omega)}$. The norm $\|\cdot\|_{\dot H^r(\Omega)}$ in the (fractional-order) Sobolev
space $\dot H^r(\Omega)$ is defined in the usual way 
via the Dirichlet eigenfunctions of the self-adjoint elliptic operator $-\nabla\cdot(\kappa \nabla)$ on $\Omega$, see \cite{thomee2006}. Therein, for a constant $\kappa,$ it is proved 
that when $r$ is a nonnegative integer, $\dot H^r(\Omega)$ consists of all functions $v$ in $ H^r(\Omega)$ which satisfy the boundary conditions $\Delta^j v=0$ on $\partial \Omega$ for $j<r/2$,  and that the norms $\|\cdot\|_{\dot H^r(\Omega)}$ and $\|\cdot\|_r$ are equivalent.

Regularity properties of the solution $u$ of the time-fractional problem \eqref{main} play a key role in the error analysis of numerical methods, particularly, since $u$ has singularity near $t=0$, even for smooth given data. From \cite[Theorems 11, 12 and 13]{McLeanetal2020} or \cite[Theorem 3]{McLeanMustapha2022}, we recall the following regularity results: for $r\in [0,2]$, $u$ and its time derivative satisfy
\begin{equation}\label{eq:reg}
\|u(t)\|_2+t\|u'(t)\|_2 \le Ct^{-\alpha(2-r)/2}\|u_0\|_{\dot H^r(\Omega)},\quad t>0,
\end{equation} 
 where $'$ denotes the time partial derivative. Properties in \eqref{eq:reg} are used in the proofs of the estimates \eqref{es1-cd}, \eqref{es3-cd}, as well as in the proof of first and third estimates of Theorem \ref{thm:mixed-sm-f}. For the error bound of the secondary variable, to show a second order of accuracy, we require that 
\begin{equation}\label{eq:reg-2}
\|u(t)\|_3\le Ct^{-\alpha(3-r)/2}\|u_0\|_{\dot H^r(\Omega)},\quad r\in [1,2].
\end{equation} 
This additional assumption is used in \eqref{es2-cd} and in the proof of the second estimate of Theorem \ref{thm:mixed-sm-f}. When $\bF \equiv 0,$ \eqref{eq:reg-2} is proved in  \cite[Theorem 4.1]{McLean2010}.

Nevertheless, the above regularity estimate holds in the limiting case $\alpha=1$, when the model problem \eqref{main} reduces to the classical Fokker-Planck PDE; see \cite[Lemmas 3.2 and 4.4]{thomee2006}. In \eqref{eq:reg} and throughout the paper, $C$ denotes a generic positive constant that may depend on $T$, $\Omega$, $\kappa$, and $\bF$. 

We show next some properties of  $\I^{\alpha}$ for later use in the subsequent sections. 
\begin{lemma} If the function $\varphi \in L^2((0,t), L^2(\Omega))$, then
\begin{equation}\label{eq: nu 1}
\int_0^t \|\I^\alpha \varphi\|^2\,ds \le \omega_{\alpha+1}(t)\int_0^t \omega_\alpha(t-s)\ \I(\|\varphi\|^2)\, ds,\quad{\rm for~any}~~0< \alpha \le 1,
\end{equation} 
\begin{equation}\label{eq: nu mu}
\int_0^t \|\I^\alpha \varphi\|^2\,ds \le Ct^{2(\alpha-\mu)}\int_0^t \|\I^\mu \varphi\|^2\,ds,\quad{\rm for~any}~~0\le \mu\le \alpha \le 1,
\end{equation} 
\begin{equation}\label{step1tech}
\I^\alpha(\|\I\varphi\|^2)(t)\le \omega_{2-\nu}(t) \I^{1+\alpha-\nu}(\|\I^\nu\varphi\|^2)(t),\quad{\rm for~any}~~0\le \alpha,\nu \le 1\,.
\end{equation}
\end{lemma}
\begin{proof}
The proof of \eqref{eq: nu 1} is in \cite[Lemma 2.3]{LMM-2016}.  For the proof of \eqref{eq: nu mu}, we refer to \cite[Lemma 3.1]{LMM-2018}. To show \eqref{step1tech},  we use first that $\I\varphi(t)=\I^{1-\nu}(\I^\nu\varphi)(t)$ for $0\le \nu \le 1,$ and notice that 
\[\|\I\varphi(t)\|=\|\I^{1-\nu}(\I^\nu\varphi)(t)\|\le \I^{1-\nu}(\|\I^\nu\varphi\|)(t)\,,\]
and hence, applying the Cauchy-Schwarz inequality and integrating, we obtain
\[
    \|\I\varphi(t)\|^2
\le \int_0^t\omega_{1-\nu}(t-s)\,ds\int_0^t\omega_{1-\nu}(t-s)\|\I^\nu\varphi(s)\|^2\,ds\\
=\omega_{2-\nu}(t) \I^{1-\nu}(\|\I^\nu\varphi\|^2)(t)\,.\]
Applying $\I^\alpha$ to both sides yields 
\[\I^\alpha(\|\I\varphi\|^2)(t)\le \omega_{2-\nu}(t) \I^\alpha\Big(\I^{1-\nu}(\|\I^\nu\varphi\|^2)\Big)(t)
= \omega_{2-\nu}(t) \I^{1+\alpha-\nu}(\|\I^\nu\varphi\|^2)(t),\]
and the proof of \eqref{step1tech} is completed. 
\end{proof} 

Before ending this section, we recall the following identity which follows from the generalized Leibniz formula: 
\begin{equation}\label{Leibniz-2}
\Ba(t\varphi)=t\Ba \varphi +\alpha\I^{1-\alpha}\varphi+t\omega_{1-\alpha}(t) \varphi(0)\,.
\end{equation}

\section {Mixed FEM} \label{sec:Mixed}
In this section, we consider the mixed form of problem \eqref{a1}. Introduce 
$\bs = - \kappa \nabla u+\bF u$ so that $\kappa^{-1} \bs +\nabla u=\beta u$ where $\beta=\kappa^{-1} \bF.$ Our model problem can now be formulated as
$$\Ba u+\nabla \cdot \bs=0, \quad u=0 \; \mbox{ on } \partial\Omega,\quad {\rm with}~~u(0)=u_0.$$
Let $H(div;\Omega)= \{\bv\in (L^2(\Omega))^2:\nabla\cdot\bv\in L^2(\Omega) \}$
be a Hilbert space equipped with the norm $\|\bv\|_{H(div;\Omega)} =(\|\bv\|^2+\|\nabla\cdot\bv\|
^2)^{\frac{1}{2}}$. Then, with $V=L^2(\Omega)$ and $\bV= H(div;\Omega)$, the mixed formulation of \eqref{main} is defined as follows: find $(u,\bs):(0,T]\to V\times \bV$ such that
\begin{eqnarray}\label{w1-m}
(\Ba u, v) + (\nabla\cdot \bs, v) &=&0 \quad \forall v \in V,\\
\label{w2-m}
(\kappa^{-1} \bs, \bv) - (u,\nabla\cdot \bv) &=& (\beta u, \bv)\quad \forall \bv \in \bV,
\end{eqnarray}
with $u(0)=u_0$. Note that the homogeneous Dirichlet boundary conditions of $u$ on $\partial\Omega$ are implicitly contained in \eqref{w2-m}. 

For the semidiscrete mixed formulation corresponding to \eqref{w1-m}-\eqref{w2-m}, consider a shape-regular partition ${\mathcal T}_h$ of the polygonal convex domain $\bar{\Omega}$ into triangles $K$, each with diameter $h_K$. Additionally, choose appropriate finite element subspaces $V_h$ and $\bV_h$ of $V$ and $\bV$, respectively, that satisfy the Ladyzenskaya-Babuska-Brezzi (LBB) condition. As an example, let's consider using the Raviart-Thomas (RT) spaces \cite{RT-1977} with an index $\ell\geq 0$, which are defined by:
$$
V_h=\{ v\in L^2(\Omega):\;v|_{K}\in P_{\ell}(K) \;\forall K\in {\mathcal T}_h\}
$$
and 
$$
\bV_h=\{ {\bf v} \in \bV:\;{\bf v}|_{K}\in RT_{\ell}(K) \;\forall K\in {\mathcal T}_h\},
$$
where $RT_{\ell}(K)=(P_{\ell}(K))^2+\boldsymbol xP_{\ell}(K),$ ${\ell}\geq 0$, and $P_{\ell}(K)$ is the space of polynomials of degree $\le \ell$ on $K$. For each of these mixed spaces, there exists a projection operator $\Pi_h:\bV \to \bV_h$ known as the Fortin projection, which satisfies the property $\nabla\cdot \Pi_h=P_h(\nabla\cdot)$, where $P_h:V\to V_h$ is the $L^2$-projection defined by
$(P_hv-v,v_h)=0$ for all $v_h\in V_h.$ Further, the following approximation properties hold:
\begin{equation}\label{eq:w9}
||{\bf w}-\Pi_h{\bf w}||\leq C h^r||\nabla\cdot{\bf w}||_{r-1}~~{\rm and}~~ \|v-P_hv\|\leq Ch^r\|v\|_r\; \forall 
~1\leq r\leq \ell+1.
\end{equation}
Additionally, it's worth noting that $\Pi_h$ and $P_h$ satisfy the following conditions:
\[
(\nabla\cdot(\Pi_h{\bf w}- {\bf w}),v_h)=0\; \forall v_h\in V_h\quad \text{and}\quad 
(P_hv-v, \nabla\cdot {\bf w}_h)=0 \;\forall {\bf w}_h \in \bV_h.
\]

For additional examples of these spaces, including Brezzi-Douglas-Marini spaces and Brezzi-Douglas-Fortin-Marini spaces, refer to \cite{BF-91}. Due to the limited smoothing properties of the continuous problem, we shall focus on the case where $\ell=1$, as high-order elements are not relevant, as discussed in \cite{K-2018}.

The corresponding semidiscrete mixed finite element approximation of problem \eqref{main} is to seek a pair $(u_h,\bs_h):(0,T]\to V_h\times \bV_h$ such that
\begin{eqnarray}\label{w1a-m}
(\Ba u_h, v_h)+ (\nabla\cdot \bs_h, v_h) &=& 0 \quad \forall v_h \in V_h,\\
\label{w2a-m}
(\kappa^{-1} \bs_h, \bv_h) - (u_h,\nabla\cdot \bv_h) &=& (\beta u_h, \bv_h) \quad \forall \bv_h \in \bV_h,
\end{eqnarray}
with $u_h(0)=u_{0h}$, where $u_{0h}$ an appropriate approximation of $u_0$ in $V_h$. 

As both $V_h$ and $\bV_h$ are finite-dimensional spaces, the system (\ref{w1a-m})-(\ref{w2a-m}) results in a system of time-fractional linear ODEs. Consequently, the theory developed for the linear system in \cite{KST-2006} guarantees the existence of a unique solution to the semidiscrete system.

To derive optimal errors from the above mixed finite element discretization for the case of smooth and nonsmooth initial data, we rely on the delicate energy argument approach. Since the problem has a limited smoothing property, integration in time with a $t$ type weight is an essential tool to provide optimal error estimates. This idea has been used in \cite{KMP2016} and \cite{Mustapha2017} to derive optimal error bounds of the standard continuous Galerkin method applied to time-fractional diffusion equations with Riemann--Liouville and Caputo derivatives, respectively. The approach in \cite{KMP2016,Mustapha2017} extends the argument in \cite{LR} for parabolic problems to the 
time-fractional case. A similar approach applied to mixed FEMs for parabolic problems has also been exploited in \cite{GP-2011}. For convenience, define $e_u=u_h-u$ and $e_\bs =\bs_h-\bs$, and introduce the intermediate solutions $(\bar u_h,\bbs_h):(0,T]\to V_h\times \bV_h$ defined by
\begin{eqnarray}\label{w1a-m-bar}
(\Ba \bar u_h, v_h)+ (\nabla\cdot \bbs_h, v_h) &=& 0
\quad \forall v_h \in V_h,\\
\label{w2a-m-bar}
(\kappa^{-1} \bbs_h, \bv_h) - (\bar u_h,\nabla\cdot \bv_h) &=& (\beta u, \bv_h) \quad \forall \bv_h \in \bV_h,
\end{eqnarray}
with $\bar u_h(0)=u_{0h}$. Setting $\teu = \bar u_h-u$ and $\tes = \bbs_h-\bs$, we split the errors $e_u=\teu+\beu$ and $e_\bs = \tes+\bes$, where $\beu := u_h-\bar u_h$ and $\bes :=\bs_h-\bbs_h$. Thus, from (\ref{w1-m})-(\ref{w2-m}) and (\ref{w1a-m-bar})-(\ref{w2a-m-bar}), we obtain 
\begin{eqnarray}\label{ee1}
(\Ba \teu, v_h)+ (\nabla\cdot \tes, v_h) &=& 0 \quad \forall v_h \in V_h,\\
\label{ee2}
( \kappa^{-1} \tes, \bv_h) - (\teu,\nabla\cdot \bv_h) &=& 0 \quad \forall \bv_h \in \bV_h.
\end{eqnarray}
For the proof of the next theorem, we refer to \cite[Theorems 5.4 and 5.6]{K-2018}. 
\begin{theorem} \label{thm:mixed-cd}
Let $u_0 \in \dot H^\delta(\Omega)$ with $\delta \in [0,2]$. Then, for $t>0$,
\begin{equation}\label{es1-cd}
 \|\teu(t)\|+t\|\teu'(t)\|\leq
 C h^r t^{-\alpha(r-\delta)/2}\|u_0\|_{\dot H^\delta(\Omega)}, \quad \mbox{for } \delta \in [0,2],
\end{equation}
with $r=1,2,$ and 
\begin{equation}\label{es2-cd}
\|\tes(t)\| \leq C h^2 t^{-\alpha(3-\delta)/2}\|u_0\|_{\dot H^\delta(\Omega)}, \quad \mbox{for } \delta \in [1,2],
\end{equation}
\begin{equation}\label{es3-cd}
\|\tes(t)\| \leq C h t^{-\alpha(2-\delta)/2}\|u_0\|_{\dot H^\delta(\Omega)}, \quad \mbox{for } \delta \in [0,1).
\end{equation}
\end{theorem}
The main task now is to estimate $\beu$ and $\bes$. We start by deriving a preliminary bound for $\bes$ in the next lemma.

\begin{lemma}\label{lem:1e} 
For $t >0,$ we have
\[\int_0^t \|\I^\alpha\bes\|^2 \,ds \le C \int_0^t \|\I^\alpha \teu\|^2\,ds\,.
\]
\end{lemma}
\begin{proof} 
From (\ref{w1a-m})-(\ref{w2a-m}) and (\ref{w1a-m-bar})-(\ref{w2a-m-bar}), we have 
\begin{eqnarray}\label{aa}
(\Ba \beu, v_h)+ (\nabla\cdot \bes , v_h) &=& 0 \quad \forall v_h \in V_h,\\
\label{bb}
(\kappa^{-1} \bes, \bv_h) - (\beu,\nabla\cdot \bv_h ) &=& (\beta e_u,\bv_h) \quad \forall \bv_h \in \bV_h.
\end{eqnarray}
 Applying $\I^\alpha$ to both sides of \eqref{aa}, then using  $\I^\alpha\Ba=\I\partial_t$ and  $\beu(0)=0$, 
\begin{equation} \label{eq: n} 
(\beu, v_h)+ (\nabla\cdot \I^\alpha \bes , v_h) =0\,.
\end{equation} 
Choosing $v_h=\I^\alpha \beu$ in \eqref{eq: n} and $\bv_h = \I^\alpha \bes$ in \eqref{bb} after applying $\I^\alpha$ to both sides, then adding them up and applying the Cauchy-Schwarz inequality,  
\[(\beu, \I^\alpha \beu) +\|\sqrt{\kappa^{-1}} \I^\alpha \bes\|^2 = (\beta \I^\alpha e_u, \I^\alpha \bes) \le C\|\I^\alpha e_u\|^2 + \frac12\|\sqrt{\kappa^{-1}} \I^\alpha \bes\|^2.
\]
Canceling the common terms, then integrating both sides, using  $\int_0^t (\beu, \I^\alpha \beu)\,ds\ge 0,$ and the diffusivity coefficient assumption $\kappa>c_0>0,$ we reach 
\begin{equation}\label{bes estimate}
  \int_0^t\| \I^\alpha \bes\|^2 \,ds  \le C\int_0^t \|\I^\alpha e_u\|^2\,ds\,.
\end{equation}
Applying $\I^\alpha$ to both sides of \eqref{bb} and \eqref{eq: n}, then choosing $v_h=\I^\alpha\beu$ and $\bv_h = \I^{2\alpha} \bes$. Adding them up to obtain 
\[\|\I^\alpha \beu\|^2 +(\kappa^{-1} \I^\alpha(\I^\alpha\bes), \I^\alpha \bes) = (\beta \I^\alpha e_u, \I^{2\alpha} \bes)\,.\]
Since $e_u=\beu +\teu$, $\|\I^\alpha e_u\|^2\le 2\|\I^\alpha \beu\|^2+2\|\I^\alpha \teu\|^2$. Hence, by using the above equation, we get 
\[\|\I^\alpha e_u\|^2 +2(\kappa^{-1} \I^\alpha(\I^\alpha\bes), \I^\alpha \bes) \le 2\|\I^\alpha \teu\|^2+ 2(\beta \I^\alpha e_u, \I^{2\alpha} \bes)\,.\]
Using  $2(\beta \I^\alpha e_u, \I^{2\alpha} \bes) \le \frac12\|\I^\alpha e_u\|^2 +C\|\I^{2\alpha} \bes\|^2$, then canceling the common terms and integrating both sides, and since $\int_0^t(\kappa^{-1} \I^\alpha(\I^\alpha\bes), \I^\alpha \bes)\,ds\ge 0$, we reach 
\[\int_0^t \|\I^\alpha e_u\|^2\,ds \le C\int_0^t \|\I^\alpha \teu\|^2\,ds+C\int_0^t\|\I^{2\alpha} \bes\|^2\,ds\,.\]
Using \eqref{eq: nu 1} with $\varphi= \I^{\alpha} \bes,$ we have  
\[\int_0^t \|\I^{2\alpha} \bes\|^2\,ds \le \omega_{\alpha+1}(t) \int_0^t \omega_\alpha(t-s)\int_0^s \| \I^\alpha \bes(q)\|^2dq\,ds,\]
and then, with the help of \eqref{bes estimate}, we deduce that
\[\int_0^t \|\I^\alpha \bes\|^2\,ds \le C\int_0^t \|\I^\alpha \teu\|^2\,ds + \omega_{\alpha+1}(t) \int_0^t \omega_\alpha(t-s)\int_0^s \| \I^\alpha \bes(q)\|^2dq\,ds.\]
An application of a weakly singular Gronwall inequality completes the proof. 
\end{proof}
In the next lemma, we derive an upper bound for $\beu$ and $\bes$. This bound leads to optimal convergence rates in 
the $L^2(\Omega)$-norm of $e_u$ and $e_\bs$.
\begin{lemma}\label{lem:2e} 
For $0<t\leq T$, we have
\[\|t\beu\|^2\le 
C \max_{0\le s\le t} \I^\alpha(\|s\teu\|^2)+ C\,t^{1-\alpha}\int_0^t \|\I^\alpha \teu\|^2\,ds,\]
and 
\[t^2\|\bes\|^2\leq  C \max_{0\le s\le t}\Big(\I^{\alpha}(\| \I^{1-\alpha} (s \teu)'\|^2)+ \I^{\alpha}(\|\I^{1-\alpha}\beu\|^2)(s)\Big)\,.\]
\end{lemma}
\begin{proof} 
Multiply both sides of \eqref{aa} and \eqref{bb} by $t$ and use \eqref{Leibniz-2} to find that
\begin{eqnarray}\label{aa-n}
(\Ba (t\beu), v_h)+ (\nabla\cdot (t\bes), v_h) &=& \alpha(\I^{1-\alpha}\beu,v_h), \\
\label{bb-n}
(\kappa^{-1} (t\bes), \bv_h) - (t\beu,\nabla\cdot \bv_h ) &=& (\beta (te_u),\bv_h).
\end{eqnarray}
Next choose $v_h=t\beu$ and $\bv_h=t\bes$ to obtain
$$
(\Ba (t\beu),t\beu)+\|\sqrt{\kappa^{-1}}\, (t\bes)\|^2 = (\beta (te_u),t\bes) +\alpha(\I^{1-\alpha}\beu,t\beu).
$$
By integrating \eqref{aa} in time followed by choosing $v_h=t\beu$ then choosing $\bv_h=t\I \bes$ in \eqref{bb} and adding them up. This leads to 
\begin{equation}\label{eq: 3-1}
  (\I^{1-\alpha}\beu, t\beu)= - (\kappa^{-1} t\bes, \I \bes ) + (\beta (te_u), \I \bes),
\end{equation} 
and hence, 
$$
(\Ba (t\beu),t\beu)+\|\sqrt{\kappa^{-1}}\, (t\bes)\|^2 = (\beta (te_u),t\bes)+\alpha (\beta (te_u), \I \bes) - \alpha(\kappa^{-1} t\bes, \I \bes ).
$$
Using  $\Ba(\|t\beu\|^2)\leq 2(\Ba (t\beu),t\beu)$, which is due to \cite[Corollary 1]{Alikhanov2012}, we get
\begin{equation}\label{eq: 3-2}
\Ba(\|t\beu\|^2) +2\|\sqrt{\kappa^{-1}}\, (t\bes)\|^2\leq 
C\|te_u\|^2 +\|\sqrt{\kappa^{-1}}\, (t\bes)\|^2+C\|\I\bes\|^2.
\end{equation}
Canceling the similar terms, then using $e_u=\teu+\beu$ and applying $\I^\alpha$, 
\begin{align*}
\|t\beu\|^2 &\leq 
C \I^\alpha(\|t\teu\|^2)+C\I^\alpha(\|t\beu\|^2) + C \I^\alpha(\|\I\bes\|^2)\,.
\end{align*}
By applying \eqref{step1tech} with $\nu=\alpha$, we have $\I^\alpha(\|\I\bes\|^2)(t)\le \omega_{2-\alpha}(t) \I( \|\I^\alpha\bes(s)\|^2)(t),$ and so, using Lemma~\ref{lem:1e}, we see that
\begin{align*}
\|t\beu\|^2&\le 
C \max_{0\le s\le t} \I^\alpha(\|s\teu\|^2) + C\,t^{1-\alpha}\int_0^t \|\I^\alpha \teu\|^2\,ds+C \I^\alpha\Big(\|t\beu\|^2\Big).\end{align*}
Consequently, an application of a weakly singular Gronwall inequality leads to the first desired estimate for $\beu$.
We now show the estimate for $\bes$. To do so, we operate $\Ba$ on \eqref{bb-n} so that
\begin{equation}\label{cc-n}
(\kappa^{-1} \Ba(t\bes), \bv_h) - (\Ba(t\beu),\nabla\cdot \bv_h )= (\beta \Ba(te_u),\bv_h).
\end{equation}
Next choose $v_h=\Ba(t\beu)$ in \eqref{aa-n} and $\bv_h=t\bes$ in \eqref{cc-n}, and arrive at
\begin{multline*}
 \|\Ba (t\beu)\|^2 +\frac{1}{2} \Ba\|\sqrt{\kappa^{-1}} t\bes\|^2 \\
 \leq C\left(\| \Ba (t \teu)\|^2+\|\I^{1-\alpha}\beu\|^2
 +\|\sqrt{\kappa^{-1}} t\bes\|^2\right) +\frac{1}{2}\|\Ba (t\beu)\|^2,
\end{multline*}
and thus, 
\[ \Ba\|\sqrt{\kappa^{-1}} t\bes\|^2 \leq C\left(\| \Ba (t \teu)\|^2+
\|\I^{1-\alpha}\beu\|^2
 +\|\sqrt{\kappa^{-1}} t\bes\|^2\right).\]
Applying $\I^{\alpha}$ to both sides and using $t\bes(0)=0$, we get
\[t^2\|\sqrt{\kappa^{-1}}\bes\|^2
\leq  C \max_{0\le s\le t}\Big(\I^{\alpha}(\| \I^{1-\alpha} (s \teu)'\|^2)+ C\I^{\alpha}(\|\I^{1-\alpha}\beu\|^2)(s)\Big) 
+C\I^{\alpha}\left(\|\sqrt{\kappa^{-1}} t\bes\|^2\right)\,.\]
For completing the proof of the second desired estimate, apply a weakly singular Gronwall inequality then use the assumption $\kappa>c_0>0$. 
\end{proof}

We are ready now to derive optimal error estimates for the semidiscrete mixed finite element problem with smooth and nonsmooth initial data. 

\begin{theorem} \label{thm:mixed-sm}
 Let the pairs $(u,\bs)$ and $(u_h,\bs_h)$ be the solutions of \eqref{w1-m}-\eqref{w2-m} and 
 \eqref{w1a-m}-\eqref{w2a-m}, respectively. Assume $u_0 \in \dot H^\delta(\Omega)$ with $\delta \in 
 [0,2]$, and $u_{0h}=P_hu_0$. Then the following error estimates hold for $t>0$:
\begin{equation}\label{es1-n}
 \|(u_h-u)(t)\|\leq
 C h^2 t^{-\alpha(2-\delta)/2}\|u_0\|_{\dot H^\delta(\Omega)}, \quad \delta\in [0,2],
\end{equation}
and
\begin{equation}\label{es2-n}
\|(\bs_h-\bs)(t)\| \leq C h^2 t^{-\alpha(3-\delta)/2}\|u_0\|_{\dot H^\delta(\Omega)},\quad \delta\in [1,2], 
\end{equation}
\begin{equation}\label{es3-n}
\|(\bs_h-\bs)(t)\| \leq C h t^{-\alpha(2-\delta)/2}\|u_0\|_{\dot H^\delta(\Omega)},\quad \delta\in [0,1). 
\end{equation}
\end{theorem}
\begin{proof} Using the estimate \eqref{es1-cd}, we find that for $\delta \in [0,2]$,
\[  \I^\alpha(\|t \teu\|^2) \le Ch^4\int_0^t (t-s)^{\alpha-1} s^{-\alpha(2-\delta)+2}\,ds\|u_0\|_{\dot H^\delta(\Omega)}
\le Ch^4t^{\alpha(\delta-1)+2}\|u_0\|_{\dot H^\delta(\Omega)}^2,\]
and 
\[\int_0^t \|\I^\alpha \teu\|^2\,ds \le Ch^4\int_0^t s^{\alpha\delta}\,ds\|u_0\|_{\dot H^\delta(\Omega)}
\le Ch^4t^{\alpha\delta+1}\|u_0\|_{\dot H^\delta(\Omega)}^2\,.\]
Then, by Lemma \ref{lem:2e}, we deduce that for $\delta \in [0,2]$,
$$
\|\beu\|^2\leq Ch^4t^{\alpha(\delta-1)}
\|u_0\|_{\dot H^\delta(\Omega)}^2= Ch^4t^{-\alpha(1-\delta)}\|u_0\|_{\dot H^\delta(\Omega)}^2\,.
$$
Similarly, for $\delta \in [1,2]$, we have
\[ \I^\alpha\Big(\| \I^{1-\alpha} (t \teu)'\|^2\Big)\leq Ch^4\int_0^t (t-s)^{\alpha-1}s^{-\alpha(4-\delta)+2}\,ds\,\|u_0\|_{\dot H^\delta(\Omega)}^2
    \leq Ch^4 t^{-\alpha(3-\delta)+2}\|u_0\|_{\dot H^\delta(\Omega)}^2,\]
and 
 \begin{multline*}
   \I^{\alpha}\left(\|\I^{1-\alpha}\beu\|^2\right)(t) \leq 
 \I^{\alpha}((\I^{1-\alpha}\|\beu\|)^2)(t)\\
 \le Ch^4\int_0^t(t-s)^{\alpha-1} s^{\alpha(\delta-3)+2}\,ds\|u_0\|_{\dot H^\delta(\Omega)}
\leq Ch^4t^{\alpha(\delta-2)+2}\|u_0\|_{\dot H^\delta(\Omega)}^2\,.
 \end{multline*}
Then, by Lemma \ref{lem:2e}, we deduce
$\|\bes\|^2\leq Ch^4t^{-\alpha(3-\delta)}\|u_0\|_{\dot H^\delta(\Omega)}^2,$ for $\delta \in [1,2].$
The desired estimates \eqref{es1-n} and \eqref{es2-n} follow now by using the triangle inequality and the bounds in 
Theorem~\ref{thm:mixed-cd}. The last inequality \eqref{es3-n}
can be obtained analogously by using the estimate \eqref{es1-cd} with $r=1$ and $\delta \in [0,1)$.
\end{proof}
\begin{remark}\label{remark:P0}
Considering the approximation properties given in \eqref{eq:w9}, it can be deduced that the estimates \eqref{es1-n} and \eqref{es2-n} are both of order $O(h)$ in the case of the lowest RT elements ($\ell=0$).
\end{remark}
\section{Fully discrete schemes}\label{MFEMC}
This section is devoted to the analysis of a fully discrete scheme for problem \eqref{w1-m}-\eqref{w2-m} based on a convolution quadrature (CQ) generated by the backward Euler (BE) method. Divide the time interval $[0,T]$ into $N$ equal subintervals with a time step size $\tau=T/N$, and let $t_j=j\tau$. 

The time fractional Riemann--Liouville derivative ${^R}\partial_t^{\alpha}(\varphi_n)$ can be conveniently discretized using the 
BECQ (with $\varphi_j= \varphi(t_j)$):
$$
\dop \varphi_n :=\tau^{-\alpha}\sum_{j=0}^na_{n-j}^{(\alpha)}\varphi_j,\quad\mbox{where}\; \sum_{j=0}^\infty a_j^{(\alpha)}\xi^j
=(1-\xi)^{\alpha},\quad a_j^{(\alpha)}=(-1)^j\left(\begin{array}{c}\alpha\\j\end{array}\right).
$$
We approximate the Riemann--Liouville time-fractional integral $\I^\alpha$ using the discrete operator 
$\dopm$, which is similarly defined. Note that $\partial_\tau^{-1}\varphi_n=\tau \sum_{j=0}^n\varphi_j$. Using the relation ${^R}\partial_t^{\alpha}\varphi(t)=\partial_t^{\alpha}(\varphi(t)-\varphi(0))$, the proposed fully implicit scheme for \eqref{w1-m}-\eqref{w2-m} is to find a pair 
$(U_h^n,\Sigma_h^n)\in V_h\times \bV_h$ such that for $n\geq 1$,
\begin{eqnarray}\label{w1a-BE}
(\dop U_h^n, v_h)+(\nabla\cdot \Sigma_h^n, v_h) &=& (\dop U_h^0,v_h)\;\;\;\forall v_h \in V_h,\\
\label{w2a-BE}
(\kappa^{-1}\Sigma_h^n, \bv_h) - (U_h^n,\nabla\cdot \bv_h) &=& (\beta U_h^n, \bv_h ) \;\;\;\forall \bv_h \in \bV_h,
\end{eqnarray}
with $U_h^0=P_hu_0$. Recall that $u_0\in \dot{H}^\delta(\Omega)$ and $\delta \in [0,2]$. For the error analysis, we introduce the intermediate discrete solution $(\bar U_h^n,\bar\Sigma_h^n)\in V_h\times \bV_h$ satisfying 
\begin{eqnarray}\label{sys1_a}
(\dop \bar U_h^n, v_h)+(\nabla\cdot \bar\Sigma_h^n, v_h) &=& (\dop U_h^0,v_h)\;\;\;\forall v_h \in V_h,\\
\label{sys1_b}
(\kappa^{-1}\bar\Sigma_h^n, \bv_h) - (\bar U_h^n,\nabla\cdot \bv_h) &=& (\beta  u_h(t_n), \bv_h ) \;\;\;\forall \bv_h \in \bV_h,
\end{eqnarray}
for $n\geq 1$, with $\bar U_h^0=U_h^0$. The first estimate is the next theorem is proved in~\cite[Theorem 6.2]{K-2020}. The estimate for $\bar \Sigma_h^n$ is proved in \cite[Theorem 6.3]{K-2020} for initial data $u_0\in L^2(\Omega)$, i.e., for the case $\delta=0$. For $\delta>0$, one can directly obtain the corresponding estimates from the proof of the same theorem. 

\begin{theorem} \label{thm:mixed-sm-f}
 Let $(\bar U_h^n,\bar \Sigma_h^n)$ and $(u_h,\bs_h)$ be the solutions of \eqref{sys1_a}-\eqref{sys1_b} and \eqref{w1a-m}-\eqref{w2a-m}, respectively. 
 Then the following error estimates hold for $t_n>0$:
\begin{equation*}\label{es1-d}
 \|\bar U_h^n-u_h(t_n)\|\leq
 C (\tau t_n^{\alpha\delta/2-1} + h^2 t_n^{-\alpha(2-\delta)/2})\|u_0\|_{\dot H^\delta(\Omega)},\quad \delta \in [0,2],
\end{equation*}
and
\[ \|\bar \Sigma_h^n-\bs_h(t_n)\|\leq C (\tau t_n^{\alpha(\delta-1)/2-1} + h^2 t_n^{-\alpha(3-\delta)/2})\|u_0\|_{\dot H^\delta(\Omega)},\quad \delta \in [1,2],
\]
\[ \|\bar \Sigma_h^n-\bs_h(t_n)\|\leq C (\tau t_n^{\alpha(\delta-1)/2-1}+ h t_n^{-\alpha(2-\delta)/2})\|u_0\|_{\dot H^\delta(\Omega)},\quad \delta \in [0,1).\]
\end{theorem}
The main task now is to estimate the errors $\dbeu^n := U_h^n-\bar U_h^n$ and 
$\dbes^n := \Sigma_h^n-\bar \Sigma_h^n$. From \eqref{w1a-BE}-\eqref{w2a-BE} and 
\eqref{sys1_a}-\eqref{sys1_b}, we note that $\dbeu^n$ and 
$\dbes^n$ satisfy
\begin{eqnarray}\label{sys2_a}
(\dop \dbeu^n,  v_h)+(\nabla\cdot \dbes^n, v_h) &=& 0\;\;\;\forall v_h \in V_h,\\
\label{sys2_b}
(\kappa^{-1}\dbes^n, \bv_h) - (\dbeu^n,\nabla\cdot \bv_h) &=& 
(\beta {\cal E}_u^n, \bv_h ) \;\;\;\forall \bv_h \in \bV_h,
\end{eqnarray}
where  ${\cal E}_u^n=U_h^n-u_h(t_n).$  For convenience,  we set $\dteu^n =\bar U_h^n-u_h(t_n)$ and $\dtes^n =\bar\Sigma_h^n-\sigma_h(t_n)$.
In the next lemma, we derive a preliminary bound for $\dbes$. Before, we recall the following properties (see, \cite{KM-2022} and \cite{WXZ2020}, respectively):
\begin{equation}\label{e:p1}
\frac{\tau^\alpha}{2}\sum_{j=0}^n\|\dbeu^n\|^2\leq \sum_{j=0}^n(\dopm\dbeu^n,\dbeu^n),
\end{equation}
and
\begin{equation}\label{e:p2}
\frac{1}{2}\dop\|\dbeu^n\|^2\leq (\dop \dbeu^n,\dbeu^n).
\end{equation}

\begin{lemma}\label{lem:2ed} For $n\ge 1,$ we have $\partial_\tau^{-\alpha}(\|\partial_\tau^{-1}\dbes\|^2)^n\le 
 C\,t_n^{1-\alpha}\partial_\tau^{-1}(\|\partial_\tau^{-\alpha} \dteu\|^2)^n\,.$
\end{lemma}
\begin{proof} 
Using the Cauchy-Schwarz inequality, an appropriate simplification yields
 \[\|\partial_{\tau}^{-1}\dbes^j\|^2\le 
\big(\partial_{\tau}^{\alpha-1}(\|\partial_{\tau}^{-\alpha}\dbes\|)^j\big)^2
\le C t_n^{1-\alpha}\partial_{\tau}^{\alpha-1}(\|\partial_{\tau}^{-\alpha}\dbes\|^2)^j,\]
for $j\ge 1,$ and thus, 
\begin{equation}\label{step1discrete}
 \partial_{\tau}^{-\alpha}(\|\partial_{\tau}^{-1}\dbes\|^2)^n
\le Ct_n^{1-\alpha} \partial_{\tau}^{-\alpha}\big(\partial_{\tau}^{\alpha-1}(\|\partial_{\tau}^{-\alpha}\dbes\|^2)\big)^n
= Ct_n^{1-\alpha} \partial_{\tau}^{-1} (\|\partial_{\tau}^{-\alpha}\dbes\|^2)^n.
\end{equation}
Operating with $\partial_\tau^{-\alpha}$ on \eqref{sys2_a} and using $\dop\partial_\tau^{-\alpha}\dbeu^n=\dbeu^n$, we get 
\begin{equation} \label{eq: ndiscrete} 
(\dbeu^n,v_h)+(\nabla\cdot \partial_\tau^{-\alpha}\dbes^n,v_h)= 0.
\end{equation}
Choosing $v_h=\partial_\tau^{-\alpha} \dbeu^n$ in \eqref{eq: ndiscrete}, and $\bv_h = \partial_\tau^{-\alpha} \dbes^n$ in \eqref{sys2_b} after operating with $\partial_\tau^{-\alpha}$ both sides, then adding them up and applying the Cauchy-Schwarz inequality to obtain
\[(\dbeu^n, \partial_\tau^{-\alpha} \dbeu^n) +\|\sqrt{\kappa^{-1}} \partial_\tau^{-\alpha} \dbes^n\|^2 = (\beta \partial_\tau^{-\alpha} {\cal E}_u^n, \partial_\tau^{-\alpha} \dbes^n)
    \le C\|\partial_\tau^{-\alpha} {\cal E}_u^n\|^2 + \frac12\|\sqrt{\kappa^{-1}} \partial_\tau^{-\alpha} \dbes^n\|^2.\]
Canceling the common terms, summing up over $n$, multiplying by $\tau,$ and noting \eqref{e:p1}, we deduce
\[\partial_\tau^{-1}(\|\sqrt{\kappa^{-1}}\partial_\tau^{-\alpha}\dbes\|^2)^n \leq C\partial_\tau^{-1}(\|\partial_\tau^{-\alpha} {\cal E}_u\|^2)^n\,.\]
Recalling that ${\cal E}_u=\dbeu+\dteu$, and hence, with   $g_\bs^n=\partial_\tau^{-1}(\|\partial_\tau^{-\alpha} \dbes\|^2)^n$, $g_u^n=\partial_\tau^{-1}(\|\partial_\tau^{-\alpha} \dbeu\|^2)^n$ and $q_u^n=\partial_\tau^{-1}(\|\partial_{\tau}^{-\alpha}\dteu\|^2)^n$, we have 
\begin{equation}\label{bes estimatediscrete}
g_\bs^n \leq g_u^n+q_u^n\,.
\end{equation}
Operating with $\partial_\tau^{-\alpha}$ both sides of \eqref{sys2_b} and \eqref{eq: ndiscrete}, then choosing $v_h=\partial_\tau^{-\alpha}\dbeu^n$ and $\bv_h = \partial_\tau^{-2\alpha}\dbes^n$. Adding them up to obtain
\begin{multline*}
  \|\partial_\tau^{-\alpha} \dbeu^n\|^2 +(\kappa^{-1} \partial_\tau^{-\alpha}(\partial_\tau^{-\alpha}\dbes)^n, \partial_\tau^{-\alpha} \dbes^n) = (\beta \partial_\tau^{-\alpha} {\cal E}_u^n, \partial_\tau^{-2\alpha}\dbes^n)\\
\le \frac12\|\partial_\tau^{-\alpha} \dbeu^n\|^2+ 
\|\partial_\tau^{-\alpha} \dteu^n\|^2+C\|\partial_\tau^{-2\alpha}\dbes^n\|^2\,.
\end{multline*}
Canceling the common terms, summing over $n,$ and using \eqref{e:p1}, we reach 
\begin{equation}
\label{cd-1ndiscrete}
g_u^n \le q_u^n +C\partial_\tau^{-1}(\|\partial_\tau^{-2\alpha}\dbes\|^2)^n\,.\end{equation}
Using 
\[ \partial_\tau^{-1}(\|\partial_\tau^{-2\alpha} \dbes\|^2)^n =
\partial_\tau^{-1}(\|\partial_\tau^{-\alpha}(\partial_\tau^{-\alpha}\dbes)\|^2)^n
\le Ct_n^\alpha \partial_\tau^{-1}(\partial_\tau^{-\alpha}(\|\partial_\tau^{-\alpha} \dbes\|^2))^n= Ct_n^\alpha \partial_\tau^{-\alpha}g_\bs^n,\]
and the estimates in \eqref{cd-1ndiscrete} and \eqref{bes estimatediscrete}, we deduce that $g_\bs^n \le Cq_u^n +Ct_n^\alpha \partial_\tau^{-\alpha}g_\bs^n.$ 
Then, if the time step $\tau$ is such that $C a_0^{(-\alpha)} \tau^\alpha <1,$ an application of a weakly singular discrete Gronwall inequality yields $g_\bs^n
\le C\,q_u^n\,.$ Inserting this in \eqref{step1discrete} and the proof is completed. 
\end{proof}

In the next lemma, we derive an upper bound for $\dbeu^n$ that will be used later to show the $L^2(\Omega)$-norm convergence rate for ${\cal E}_u^n$. For brevity, for a given function $f,$ we set $(t f)^n=t_nf(t_n)=t_n f^n.$
\begin{lemma}\label{lem:2ee} 
For $n\ge 1$, we have
\[\|t_n\dbeu^n\|^2 \leq C\max_{1\le j\le n} \Big(\dopm(\|(t\dteu)^j\|^2)+t_{j-1}^{1-\alpha}\partial_\tau^{-1}(\|(\partial_\tau^{-\alpha} \dteu)\|^2)^{j-1}\Big)\,.\]
\end{lemma}
\begin{proof} 
Multiply both sides of \eqref{sys2_a} and \eqref{sys2_b} by $t_n$ and use the identity
\begin{equation}\label{formula}
t_n\dop U_h^n=\dop((t U_h)^n)-\alpha\partial_\tau^{\alpha-1}U_h^{n-1},\quad n\geq 1,
\end{equation}
see \cite[(5.15)]{KM-2022}, we get
\begin{eqnarray}\label{aa-nd}
(\dop (t\dbeu)^n, v_h)+ (\nabla\cdot (t\dbes)^n, v_h) &=& \alpha(\partial_\tau^{\alpha-1}\dbeu^{{n-1}},v_h), \\
\label{bb-nd}
(\kappa^{-1} (t\dbes)^n, \bv_h) - (t_n\dbeu^n,\nabla\cdot \bv_h ) &=& (\beta (t_n {\cal E}_u^n),\bv_h).
\end{eqnarray}
Choose $v_h=t_n\dbeu^n$ and $\bv_h=t_n\dbes^n$ to get
\begin{equation}\label{ab:nd}
(\dop(t \dbeu)^n,(t \dbeu)^n)+\|\sqrt{\kappa^{-1}}(t_n\dbes^n)\|^2 =
(\beta(t_n{\cal E}_u^n), t_n\dbes^n)+\alpha (\mdop\dbeu^{{n-1}}, t_n\dbeu^n).
\end{equation}
Now, {consider \eqref{sys2_a} at the time level $t_{n-1}$}, apply $\partial_\tau^{-1}$ to both sides of the equation followed by choosing $v_h=t_n\dbeu^n$ then choosing $\bv_h=t_n\partial_\tau^{-1}\dbes^{{n-1}}$ in \eqref{sys2_b}
and adding up yields
$$(\mdop \dbeu^{{n-1}},t_n\dbeu^n)= -(\kappa^{-1} t_n\dbes^n, \partial_\tau^{-1}\dbes^{{n-1}})
+\alpha (\beta t_n{\cal E}_u^n,\partial_\tau^{-1}\dbes^{{n-1}}).$$
Substituting the new expression for $(\mdop \dbeu^{n-1},t_n\dbeu^n)$ in \eqref{ab:nd} and using the property \eqref{e:p2}, we deduce that
\[    \dop (\|(t\dbeu)^n\|^2)+2\|\sqrt{\kappa^{-1}}(t_n\dbes^n)\|^2 \leq
C\|t_n{\cal E}_u^n\|^2+\|\sqrt{\kappa^{-1}}(t_n\dbes^n)\|^2+
C{\|\partial_\tau^{-1}\dbes^{n-1}\|^2}.\]
Cancelling similar terms, applying $\dopm$ and using the estimate in Lemma~\ref{lem:2ed}, 
\[\|t_n\dbeu^n\|^2 \leq C\dopm(\|(t\dbeu)^n\|^2)+C\max_{1\le j\le n} \Big(\dopm(\|(t\dteu)^j\|^2)+t_{j-1}^{1-\alpha}\partial_\tau^{-1}(\|(\partial_\tau^{-\alpha} \dteu)\|^2)^{j-1}\Big).
\]
If the time step $\tau$ satisfies $Ca_0^{(-\alpha)}\tau^{\alpha}<1$, then an application of a weakly singular discrete Gronwall inequality yields
the desired estimate in the lemma.
\end{proof}

We are ready to show optimal error estimates for the fully discrete solution.
\begin{theorem} \label{thm:mixed-smd}
 Let $(u,\bs)$ be the solutions of \eqref{w1-m}-\eqref{w2-m}. Let $(U_h^n, \Sigma_h^n)$ be the solution of \eqref{w1a-BE}-\eqref{w2a-BE} with $u_{0h}=P_hu_0$. 
 Then, for a small time step $\tau$, we have
\begin{equation}\label{es1n}
\|U_h^n-u(t_n)\|\leq
 C (\tau t_n^{\alpha\delta/2-1} + h^2 t_n^{-\alpha(2-\delta)/2})\|u_0\|_{\dot H^\delta(\Omega)},\quad \delta \in [0,2].
\end{equation}
\end{theorem}
\begin{proof}
 Using the first estimate in Theorem \ref{thm:mixed-sm-f} and the fact that 
 $\dopm(t_n^\beta)\leq C t_n^{\alpha+\beta}$ for $\beta>-1$, we find that 
$$
\dopm(\|(t\dteu)^n\|^2)\leq C\left(\tau^2 t_n^{\alpha(\delta+1)} + h^4 t_n^{-\alpha(1-\delta)+2}\right),
$$
and
$$
\partial_\tau^{-1}(\|\partial_\tau^{-\alpha} \dteu\|^2)^n\leq C\left(\tau^2 t_n^{\alpha(\delta+2)-1} + h^4 t_n^{\alpha\delta+1}\right),
$$
for $\delta \in [0,2]$. Then, by Lemma \ref{lem:2ee}, we see that 
\begin{equation}\label{t1-a}
\|t_n\dbeu^n\|^2\leq C\left(\tau^2 t_n^{\alpha(\delta+1)} + h^4 
t_n^{-\alpha(1-\delta)+2}\right).
\end{equation}
The desired estimate \eqref{es1n} follows now by using the triangle inequality and the first estimate in Theorem \ref{thm:mixed-sm-f}.
\end{proof}
\begin{remark}
To show the estimate for $\dbes^n$, we operate $\dop$ on \eqref{bb-nd} so that
\begin{equation*}\label{cc-nd}
(\kappa^{-1} \dop(t\dbes)^n, \bv_h) - (\dop(t\dbeu)^n,\nabla\cdot \bv_h )= (\beta \dop(t{\cal E}_u)^n,\bv_h).
\end{equation*}
Next, we choose $\bv_h=t_n\dbes^n$ and $v_h=\dop(t\dbeu)^n$ in \eqref{aa-nd} to arrive at
\begin{multline*}
 \|\dop (t\dbeu)^n\|^2 +\frac{1}{2} \dop\|\sqrt{\kappa^{-1}} t_n\dbes^n\|^2 \\ \leq 
 C\left(\| \dop (t \dteu)^n\|^2+\|\partial_\tau^{\alpha-1}\dbeu^{{n-1}}\|^2
 +\|\sqrt{\kappa^{-1}} t_n\dbes^n\|^2\right) +\frac{1}{2}\|\dop (t\dbeu)^n)\|^2.
\end{multline*}
Canceling the similar terms, then applying $\partial_\tau^{-\alpha}$ to both sides, 
\begin{multline*}
\|\sqrt{\kappa^{-1}} t_n\dbes^n\|^2 \leq 
 C\max_{1\le j\le n} \Big(\dopm\Big(\| \dop (t \dteu)\|^2\Big)^j+\dopm\Big(\|\partial_\tau^{\alpha-1}\dbeu\|^2\Big)^{j-1}\Big)\\
+C\dopm\left(\|\sqrt{\kappa^{-1}} (t\dbes)^n\|^2\right)\,.
\end{multline*}
Thus, if the time step $\tau$ is such that $Ca_0^{(-\alpha)}\tau^{\alpha}<1$, then an application of a weakly singular discrete Gronwall inequality  yields 
\[t_n^2\|\sqrt{\kappa}\dbes^n\|^2\leq C\,\max_{1\le j\le n} \Big(\dopm\Big(\| \dop (t \dteu)\|^2\Big)^j+\dopm\Big(\|\partial_\tau^{\alpha-1}\dbeu\|^2\Big)^{j-1}\Big),\quad n\geq 1.\]
The second term $\|\mdop\dbeu^{j-1}\|$ on the right-hand side can be readily bounded using the estimate \eqref{t1-a}. However, deriving an appropriate bound for the first term $\dopm\left(\| \dop (t\dteu)\|\right)^j$ which preserves the convergence rate seems to be a challenging task.

\end{remark}

\section{Numerical experiments}\label{Numerical} 
In this section, we present numerical tests to validate the theoretical predictions discussed in the convergence theorems (Theorems \ref{thm:mixed-sm} and \ref{thm:mixed-smd}) for different choices of the fractional order $\alpha$. In the model problem \eqref{main}, we choose $T=0.5$, the spatial domain $\Omega$ is a unit square, $\kappa=1$, and the driving vector force  $\bF(x,y)=\Big[\begin{matrix} &x\\ &y\end{matrix}\Big]$.  We consider three numerical tests with smooth and nonsmooth initial data $u_0$.

The calculations are performed on a uniform (spatial) triangular grid. The construction of the grid involves dividing the unit square into two triangles by drawing a diagonal from (0, 0) to (1, 1). Then, a sequence of red refinements is applied to generate a sequence of shape-regular grids.
For the mixed FEM, we perform the computation using the Raviart-Thomas finite elements $RT0/P0$ and $RT1/P1^{\rm disc}$, where $P0$ and $P1^{\rm disc}$ denote the sets of piecewise constant and linear functions, respectively.
All computations are performed using the finite element package written in Julia, which provides the necessary tools for finite element analysis. \\

{\bf Example 1.} We consider a test problem with the following initial data
\begin{align*}
	u_0(x,y)=x(1-x)y(1-y)\,.
\end{align*}
Since the exact solution is not available, we compute a reference solution $u^*$ on the refinement level 6, representing a $2^7 \times 2^7=128\times 128$ spatial grid. To examine the spatial discretization error, we choose a (fixed) small  $\tau$ so that the error incurred by the spatial discretization is dominant. We compute the $L^2(\Omega)$-norm of the errors at the final time $T$. The numerical results are given in Table~\ref{tab:ref_rt0} for the case of the lowest-order Raviart-Thomas finite elements $RT_0/P_0$ for two different values of $\alpha$; $\alpha=0.3$ and $\alpha=0.7$. The table shows a convergence rate $O(h)$ for both the scalar variable $u$ and the flux variable $\bs$, which is consistent with our observation in Remark~\ref{remark:P0}. For the finite elements $RT_1/P_1^{\rm disc}$, the results are presented in Table~\ref{tab:ref_rt1}. We observe $O(h^2)$-rates for both variables $u$ and $\bs$. As the initial data $u_0\in \dot H^{5/2-\epsilon}(\Omega)$ for $\epsilon>0$, the numerical results agree with the achieved theoretical estimates in Theorem~\ref{thm:mixed-sm}.

\begin{table}[htb!!]
	\begin{center}
  \begin{tabular}{|c|cc|cc|cc|cc|}
 \hline
 & \multicolumn{4}{|c}{$\alpha=0.3$} &\multicolumn{4}{|c|}{$\alpha=0.7$}\\
 \hline
level&$\|u^*-u_h\|$&OC&$\|\bs^*-\bs_h\|$&OC
&$\|u^*-u_h\|$&OC&$\|\bs^*-\bs_h\|$&OC\\
\hline
      1&1.074e-02&     &4.389e-02&   &1.724e-03&     &5.984e-03& \\
      2&4.315e-03&1.32&1.684e-02&1.38&8.772e-04&0.97 &2.311e-03&1.37\\
      3&2.144e-03&1.01&8.325e-03&1.07&4.388e-04&0.99 &1.146e-03&1.01\\
      4&1.044e-03&1.04&4.105e-03&1.02&2.230e-04&0.98 &5.659e-04&1.02\\
      5&5.299e-04&0.98&2.011e-03&1.03&1.101e-04&1.02 &2.781e-04&1.02\\
			\hline
		\end{tabular}
 \caption{$L^2$-errors and spatial order of convergence  (OC)  for Example 1 using $RT_0/P_0$.  
     }\label{tab:ref_rt0}
     \end{center}
\end{table}

\begin{table}[htb!!]
	\begin{center}
			\begin{tabular}{|c|cc|cc|cc|cc|}
   \hline
 & \multicolumn{4}{|c}{$\alpha=0.3$} &\multicolumn{4}{|c|}{$\alpha=0.7$}\\
 \hline
    level&$\|u^*-u_h\|$&OC&$\|\bs^*-\bs_h\|$&OC
    &$\|u^*-u_h\|$&OC&$\|\bs^*-\bs_h\|$&OC\\
    \hline
       1&1.059e-03&    &4.906e-03&    &6.977e-04&    &5.437e-04&  \\
       2&2.118e-04&2.32&1.222e-03&2.01&1.138e-04&2.62&1.085e-04&2.32\\
       3&5.378e-05&1.98&2.298e-04&2.41&2.617e-05&2.12&2.765e-05&1.97\\
       4&1.323e-05&2.02&5.705e-05&2.01&6.602e-06&1.99&6.828e-06&2.02\\
			\hline
		\end{tabular}
	    \caption{$L^2$-errors and spatial order of convergence  (OC) for Example 1 using $RT_1/P_1^{\rm dc}$
    .}\label{tab:ref_rt1}
    \end{center}
\end{table}

{\bf Example 2 (non-homogeneous).} In this example, the initial data $u_0$ is a two-dimensional continuous piecewise linear function (having a hat shape) such that $u_0(1/2,1/2)=1/2$ and $u_0=0$ at the four corners of $\Omega.$ So,  $u_0\in \dot H^{3/2-\epsilon}(\Omega)$ for $\epsilon>0$. We choose a source term $g$, that is the right-hand side of \eqref{main} is $g$ instead of $0$ in this case,  so that the exact solution is 
\begin{equation*}\label{eq: u series}
  u(x,y,t)=16\sum_{m,n=0}^\infty (-1)^{mn}(\lambda_m\lambda_n)^{-2}\sin( \lambda_m x)\sin( \lambda_n y) E_{\alpha,1}(-\lambda_{mn} t^{\alpha}),
\end{equation*}
where $E_{\mu,\beta}(t):=\sum_{p=0}^\infty\frac{t^p}{\Gamma(\mu p+\beta)}$ is the Mittag-Leffler function, with parameters $\mu,\beta >0,$ $\lambda_m= (2m+1)\pi$, and $\lambda_{mn}=[(2m+1)^2+(2n+1)^2]\pi^2.$ In this case, 
\[ g(x,y,t)=16t^{\alpha-1} \sum_{m,n=0}^\infty (-1)^{mn}(\lambda_m\lambda_n)^{-2}
 \varphi_{mn}(x,y) E_{\alpha,\alpha}(-\lambda_{mn} t^{\alpha}),\]
where 
\[\varphi_{mn}(x,y)=\Big(\lambda_m x \cos( \lambda_m x)+2\sin( \lambda_m x)\Big)\sin( \lambda_n y)  +\lambda_n y\, \sin( \lambda_m x) \cos( \lambda_n y)\,.\]
To find the errors in space and time, we use the standard strategy which is to consider the time step size small enough to find the convergence order in space and vice versa. In order to keep the error contribution from time small, we set the time step length $\tau=1/1200$. Table~\ref{tab-k2} shows the errors and convergence rates in space for $RT_1/P_1^{\rm disc}$ elements. The expected order of convergence $\mathcal{O}(h^2)$ is obtained for the scalar variable $u$ and the corresponding flux $\bs$. For completeness, we present in Figure~\ref{fig:1} the errors computed using the $RT_0/P_0$ and $RT_1/P_1^{\rm disc} $ elements for a fixed $\alpha=0.5$ and different refinement levels.

To illustrate the convergence rate in time, we excluded the error in space. The numerical tests are performed with $RT_1/P_1^{\rm disc}$ and a spatial mesh consists of $128\times 128$ cells. The calculations were done with time step lengths $\tau=0.5/N$ for different choices of $N$. In Table~\ref{tab-k2_time}, the errors and convergence orders are listed for the scalar variable $u$. A first-order convergence rate is observed, which aligns with theoretical results in Theorem~\ref{thm:mixed-smd}.
\begin{table}[htb!!]
	\begin{center}
		\begin{tabular}{|c|cc|cc|cc|cc|}
\hline
 & \multicolumn{4}{|c}{$\alpha=0.3$} &\multicolumn{4}{|c|}{$\alpha=0.7$}\\
 \hline
level&$\|u-u_h\|$&OC&$\|\bs-\bs_h\|$&OC
&$\|u-u_h\|$&OC&$\|\bs-\bs_h\|$&OC\\
\hline
1&3.704e-03&     &1.565e-02&     & 2.256e-03&     &8.893e-03&    \\
2&9.715e-04&1.931&3.865e-03&2.018& 5.255e-04&2.101&2.192e-03&2.020 \\
3&2.147e-04&2.178&9.667e-04&2.000& 1.286e-04&2.031&4.853e-04&2.175 \\
4&5.317e-05&2.013&2.394e-04&2.014& 3.140e-05&2.033&1.194e-04&2.023 \\
5&1.437e-05&1.887&6.651e-05&1.848& 8.263e-06&1.926&3.061e-05&1.963 \\
 \hline
		\end{tabular}
 \caption{
 $L^2$-errors and spatial order of convergence (OC) for Example 2 using $RT_1/P_1^{\rm dc}$.}
 \label{tab-k2}
	\end{center}
\end{table}

	\begin{figure}[htb!]
		\begin{center}
			\includegraphics[width=0.48\textwidth]{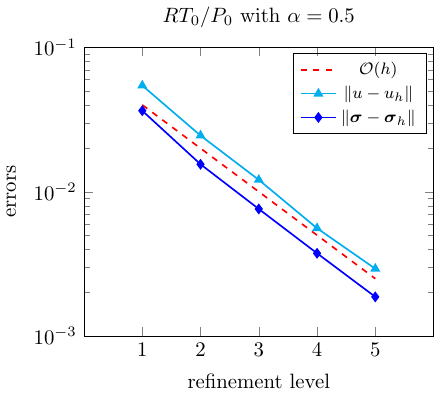}
            \includegraphics[width=0.48\textwidth]{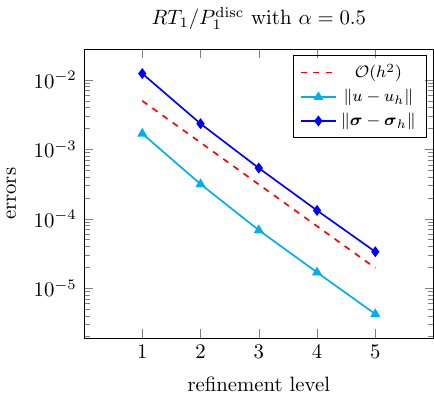}
		\end{center}
	    \caption{$L^2$-errors for Example 2 using $RT_0/P_0$ (left) and $RT_1/P_1^{\rm disc} (right)$ for fixed $\alpha=0.5$.}
     \label{fig:1}
	\end{figure}

 \begin{table}[htb!!]
 	\begin{center}
 		\begin{tabular}{|c|cc|cc|}
   \hline
 & \multicolumn{2}{|c}{$\alpha=0.3$} &\multicolumn{2}{|c|}{$\alpha=0.7$}\\
\hline
 $N$&$\|u-u_h\|$&OC&$\|u-u_h\|$&OC\\ 
 \hline                                                 
 10 & 3.1744041e-02 &    & 1.679786e-02&   \\
 20 & 1.7003601e-02 &0.901 & 7.812162e-03&1.105\\
 40 & 8.0542173e-03 &1.078 & 3.287313e-03&1.249\\
 80 & 4.3481248e-03 &0.889 & 1.595068e-03&1.043\\
 160 & 2.1347824e-03 &1.026 & 8.042011e-04&0.988\\
 \hline
 \end{tabular}
 \caption{
 $L^2$-errors and temporal order of convergence (OC)  for Example 2
.}\label{tab-k2_time}
 \end{center}
 \end{table}

{\bf Example 3 (non-homogeneous less smooth).} We choose the initial data $u_0$ to be a two-dimensional piecewise constant (discontinuous) function such that $u_0(x,y)=1$ for $1/4\le x,y\le 3/4$, and zero elsewhere. So, 
$u_0\in \dot H^{1/2-\varepsilon}(\Omega)$, for $\varepsilon>0$. We choose a source term $g$ and boundary conditions such that the exact solution is 
\begin{equation*}\label{eq: u series 3-k}
  u(x,y,t)=8\sum_{m,n=0}^\infty (-1)^{\lceil m/2 \rceil \lceil n/2 \rceil} (\lambda_m\lambda_n)^{-1}\sin( \lambda_m x)\sin( \lambda_n y) E_{\alpha,1}(-\lambda_{mn} t^{\alpha})\,.
\end{equation*}
The source term in that case is
\[ g(x,y,t) = 8t^{\alpha-1} \sum_{m,n=0}^\infty(-1)^{\lceil m/2 \rceil \lceil n/2 \rceil} (\lambda_m\lambda_n)^{-1}
 \varphi_{mn}(x,y) E_{\alpha,\alpha}(-\lambda_{mn} t^{\alpha}).\]

 \begin{table}[htb!!]
	\begin{center}
 		\begin{tabular}{|c|cc|cc|cc|cc|}
\hline
 & \multicolumn{4}{|c}{$\alpha=0.3$} &\multicolumn{4}{|c|}{$\alpha=0.7$}\\
 \hline
level&$\|u-u_h\|$&OC&$\|\bs-\bs_h\|$&OC
&$\|u-u_h\|$&OC&$\|\bs-\bs_h\|$&OC\\
\hline
1&5.587e-03&      &4.003e-02&         &2.493e-03&    &1.784e-02&   \\
2&1.311e-03&2.091 &1.084e-02& 1.885 &6.549e-04&1.928 &5.182e-03&1.783\\
3&3.019e-04&2.119 &3.219e-03& 1.752 &1.570e-04&2.061 &1.626e-03&1.672\\
4&7.956e-05&1.924 &9.967e-04& 1.691 &3.932e-05&1.998 &5.356e-04&1.603\\
 \hline
		\end{tabular}
 \caption{
 $L^2$-errors and temporal order of convergence (OC) for Example 3 using $RT_1/P_1^{\rm dc}$ .}\label{tab-k3:rt1}
\end{center}
\end{table}

This example studies the convergence order with respect to space. We fix 
the time step length $\tau=6.25\times 10^{-4}$ and perform the simulations for the finite element 
$RT_1/P_1^{\rm disc}$ on different refinement levels. The errors and the convergence orders using $\alpha=0.3$ and $\alpha=0.7$ are presented in Table~\ref{tab-k3:rt1}. One can observe the predicted convergence order $\mathcal{O}(h^2)$ for the errors estimated in Theorems~\ref{thm:mixed-sm} and \ref{thm:mixed-smd} for the scalar variable $u$. 
By contrast, the convergence rate is of order $\mathcal{O}(h^{1.6})$ for the flux variable $\bs$, which may be seen as expected results. However, as the initial data $u_0\in \dot H^{1/2-\varepsilon}(\Omega)$ for $\varepsilon>0$, the numerical results show a positive effect on the convergence rate.
Indeed, by \eqref{es2-n}-\eqref{es3-n}, and interpolation between $\dot H^{1}(\Omega)$ and $L^2(\Omega)$, it follows that
$$
\|\bs-\bs_h\|\leq Ch^{3/2-\varepsilon}t^{-\alpha(2+\varepsilon)/2}\|u_0\|_{\dot H^{1/2-\varepsilon}(\Omega)},
$$ confirming our theoretical estimates.

\end{document}